\documentclass[12pt,fleqn]{amsart}
\usepackage{amsfonts,amssymb,verbatim,amsmath,amsthm,latexsym,textcomp,amscd}
\usepackage{latexsym,amsfonts,amssymb,epsfig,verbatim}
\usepackage{amsmath,amsthm,amssymb,latexsym,graphics,textcomp}
\usepackage{graphicx}
\usepackage{color}
\usepackage{url}

\usepackage[margin=1in]{geometry}
\usepackage{marginnote}

\usepackage{mathptmx}
\usepackage[T1]{fontenc}
\usepackage{comment}
\usepackage{euscript}

\usepackage{latexsym,epsfig,verbatim}
\usepackage{amsmath,amsthm,amssymb}
\usepackage{colonequals}
\usepackage{units}
\usepackage{leftidx}

\usepackage{url}
\usepackage{color}
\usepackage[colorlinks,citecolor=blue,linkcolor=red!80!black]{hyperref}

\usepackage{accents}

\usepackage{tikz}
\usetikzlibrary{matrix,arrows,decorations.pathmorphing}

\begin{document}

\theoremstyle{plain}
\newtheorem{theorem}{Theorem}[section]
\newtheorem{lemma}[theorem]{Lemma}
\newtheorem{corollary}[theorem]{Corollary}

\newtheorem{conjecture}[theorem]{Conjecture}
\newtheorem{assumption}[theorem]{Assumption}
\newtheorem{proposition}[theorem]{Proposition}
\newtheorem{observation}[theorem]{Observation}
\newtheorem{fact}[theorem]{Fact}
\newtheorem{claim}[theorem]{Claim}
\newtheorem{prop}[theorem]{Proposition}

\newtheorem{cor}[theorem]{Corollary}
\newtheorem{definition}[theorem]{Definition}
\newtheorem{defn}[theorem]{Definition}
\newtheorem{conj}[theorem]{Conjecture}

\newtheorem{defth}[theorem]{Definition-Theorem}
\newtheorem{obs}[theorem]{Observation}
\newtheorem{rmark}[theorem]{Remark}
\newtheorem{qn}[theorem]{Question}
\newtheorem{theo}[theorem]{Theorem}
\newtheorem{thmbis}{Theorem}
\newtheorem{dfn}[theorem]{Definition} 
\newtheorem{defi}[theorem]{Definition} 
\newtheorem{coro}[theorem]{Corollary}
\newtheorem{hypo}[theorem]{Hypothesis}
\newtheorem{corbis}{Corollary}
\newtheorem{propbis}{Proposition} 
\newtheorem*{prop*}{Proposition} 
\newtheorem{lem}[theorem]{Lemma} 
\newtheorem{lembis}{Lemma} 
\newtheorem{claimbis}{Claim} 

\newtheorem{factbis}{Fact} 
\newtheorem{qst}[theorem]{Question} 
\newtheorem{qstbis}{Question} 
\newtheorem{pb}[theorem]{Problem} 
\newtheorem{pbbis}{Problem} 
\newtheorem{question}[theorem]{Question}
\newtheorem{rem}[theorem]{Remark}
\newtheorem{rmk}[theorem]{Remark}
\newtheorem{remark}[theorem]{Remark}
\newtheorem{example}[theorem]{Example}
\newtheorem{eg}[theorem]{Example}
\newtheorem{notation}[theorem]{Notation}
%repeated theorems

\newtheorem*{case}{Case}

\newcommand{\define}[1]{\textbf{#1}}
\newcommand{\ext}{\mathrm{ext}}

%%%%%%%%%%%%%%%%%%%%%%%%%%%%%%%%%%%%%%%%%%%%
%%%%%    Notation Macros    %%%%%%%%%%%%%%%%
%%%%%%%%%%%%%%%%%%%%%%%%%%%%%%%%%%%%%%%%%%%%
%basic notation
\newcommand{\R}{\mathbb{R}}
\newcommand{\Z}{\mathbb{Z}}
\newcommand{\Q}{\mathbb{Q}}
\newcommand{\N}{\mathbb{N}}

\newcommand{\floor}[1]{\left\lfloor#1\right\rfloor}
\newcommand{\ceil}[1]{\left\lceil#1\right\rceil}
\newcommand{\norm}[1]{\left\Vert#1\right\Vert}
\newcommand{\abs}[1]{\left\vert#1\right\vert}
\newcommand{\inv}{^{-1}}
\newcommand{\hhat}{\widehat}
\newcommand{\boundary}{\partial}
\newcommand{\C}{{\mathbb C}}
\newcommand{\integers}{{\mathbb Z}}
\newcommand{\natls}{{\mathbb N}}
\newcommand{\bbN}{{\mathbb N}}
\newcommand{\hypp}{{{\mathbb H}^2}}
\newcommand{\hyps}{{{\mathbb H}^3}}
\newcommand{\ratls}{{\mathbb Q}}
\newcommand{\reals}{{\mathbb R}}
\newcommand{\bbR}{{\mathbb R}}
\newcommand{\lhp}{{\mathbb L}}
\newcommand{\tube}{{\mathbb T}}
\newcommand{\cusp}{{\mathbb P}}
\newcommand\AAA{{\mathcal A}}
\newcommand\BB{{\mathcal B}}
\newcommand\CC{{\mathcal C}}
\newcommand\DD{{\mathcal D}}
\newcommand\EE{{\mathcal E}}
\newcommand\FF{{\mathcal F}}
\newcommand\GG{{\mathcal G}}
\newcommand\HH{{\mathcal H}}
\newcommand\II{{\mathcal I}}
\newcommand\JJ{{\mathcal J}}
\newcommand\KK{{\mathcal K}}
\newcommand\LL{{\mathcal L}}
\newcommand\MM{{\mathcal M}}
\newcommand\NN{{\mathcal N}}
\newcommand\OO{{\mathcal O}}
\newcommand\PP{{\mathcal P}}
\newcommand\QQ{{\mathcal Q}}
\newcommand\RR{{\mathcal R}}
\newcommand\SSS{{\mathcal S}}
\newcommand\TT{{\mathcal T}}
\newcommand\UU{{\mathcal U}}
\newcommand\VV{{\mathcal V}}
\newcommand\WW{{\mathcal W}}
\newcommand\XX{{\mathcal X}}
\newcommand\YY{{\mathcal Y}}
\newcommand\ZZ{{\mathcal Z}}
\newcommand{\Out}{\mathrm{Out}}
\newcommand{\Mod}{\mathrm{Mod}}
\newcommand{\hull}{\mathrm{hull}}

% Typesetting etc...
\newcommand{\co}{\colon\thinspace}

% Coarse language
\newcommand{\emul}[1]{\overset{#1}{\asymp}}
\newcommand{\gmul}[1]{\overset{#1}{\succ}}
\newcommand{\lmul}[1]{\overset{#1}{\prec}}

% Mapping class group related

\newcommand{\free}{\mathbb{F}} % the base free group (of rank n)
\newcommand{\factor}{{\EuScript F}} % the factor complex
\newcommand{\F}{\factor} % shorthand for the factor complex
\newcommand{\fc}{\factor} % more shorthand for the factor complex
\renewcommand{\int}{\mathcal{I}} % the intersection graph

\newcommand{\geod}{\overleftrightarrow}
\newcommand{\I}{\mathbf{I}}
\newcommand{\Ipl}{\I_+}
\newcommand{\Imin}{\I_-}
\newcommand{\J}{\mathbf{J}}
\newcommand{\Jpl}{\J_+}
\newcommand{\Jmin}{\J_-}
\newcommand{\cay}[2]{\mathrm{Cay}({#2}, {#1})} % The #1--Cayley graph of a f.g group #2=<#1>
\newcommand{\diam}{{\rm diam}}

\newcommand{\Hl}{\{H_\lambda\}_{\lambda \in \Lambda}}
\newcommand{\he}{hyperbolically embedded}
\newcommand{\gen}[1]{\left\langle#1\right\rangle}
\newcommand{\g} {\ensuremath {\gamma}}
\newcommand{\ga}{\Gamma}
\newcommand{\G}{\ga(G, X\sqcup \mathcal H)}
\newcommand{\h}{\hookrightarrow_{h}}
\newcommand{\cH}{\mathcal{H}}
\newcommand{\e}{\varepsilon }
\newcommand{\nei}[2]{N_{#2}(#1)}
\renewcommand{\d}{{\rm d}}
\newcommand{\dx}{{\rm d}_X}
\newcommand{\dhau}{{\rm d}_{\rm Haus}}

\title{Intersection properties of stable subgroups and bounded cohomology}

\author{Yago Antol\'{i}n}
\address{Departamento de Matem\'{a}ticas, Universidad Aut\'{o}noma de Madrid and Instituto de Ciencias Matem\'{a}ticas, CSIC-UAM-UC3M-UCM }
\email{yago.anpi@gmail.com}

\author{Mahan Mj} 

\address{School of Mathematics, Tata Institute of Fundamental Research, Homi Bhabha Road, Mumbai-400005,  India}
\email{mahan.mj@gmail.com; mahan@math.tifr.res.in}

\author{Alessandro Sisto}
\address{Department of Mathematics, ETH Zurich, 8092 Zurich, Switzerland}
\email{sisto@math.ethz.ch}

\author{Samuel J. Taylor}
\address{Department of Mathematics,
Yale University,
10 Hillhouse Ave,
New Haven, CT 06520, U.S.A.}
\email{s.taylor@yale.edu}

\thanks{The first named author is supported by the Juan de la Cierva grant IJCI-2014-22425, and acknowledges partial support from the
Spanish Government through grant number MTM2014-54896-P.
The second author  is partially supported by  a DST J C Bose Fellowship and the fourth author is partially supported by  NSF DMS-1400498.
The second, third and fourth authors were supported in part by the National Science Foundation 
	under Grant No. DMS-1440140 at the Mathematical Sciences Research Institute in Berkeley
	during Fall 2016 program in Geometric Group Theory.
	The first, third and fourth author would like to thank the Isaac Newton Institute for Mathematical Sciences, Cambridge, for support and hospitality during the programme ``Non-Positive Curvature Group Actions and Cohomology'' where work on this paper was undertaken. This work was supported by EPSRC grant no EP/K032208/1.}

\date{\today}

%\subjclass[2010]{20F65, 20F67 (Primary);   22E40  (Secondary)}

\begin{abstract} 
We show that a finite collection of stable subgroups of a finitely generated group has finite height, finite width and bounded packing.  We then use knowledge about intersections of conjugates to characterize finite families of quasimorphisms on hyperbolically embedded subgroups that can be to simultaneously extended to the ambient group.
\end{abstract}

\maketitle

%%%%%%%%%%%%%%%%%%%%%%%%%%%
\section{Introduction}
\label{sec:intro}
%%%%%%%%%%%%%%%%%%%%%%%%%%%
There is a well-developed theory of convex cocompact subgroups of Kleinian groups. This  has been generalized, first
to mapping class groups \cite{farb-mosher, kl, ham-cc, mahan-sardar}, and then to $\Out(F_n)$ \cite{dt1, ham-coco, ADT} over the last several years.
Further extending these notions, a general theory of 
subgroup stability
was introduced 
by Durham and Taylor \cite{DT15}.

One of the aims of this paper is to extend some well-known
intersection properties of quasiconvex
subgroups of hyperbolic or relatively hyperbolic  groups  \cite{GMRS, hruska-wise}
to the context of stable subgroups  of finitely generated groups:

\begin{theorem} \label{th:intro_main}
Let $H_1, \cdots, H_l$ be stable subgroups of a finitely generated group.
Then the collection $\HH = \{H_1, \cdots, H_l \}$ satisfies the following:
\begin{enumerate}
\item $\HH$ has finite height.
\item  $\HH$ has finite width.
\item $\HH$ has bounded packing.
\end{enumerate}
\end{theorem} 

The proof of finiteness of height follows the same line of argument as \cite{GMRS} and has already been used to observe finiteness of height in convex cocompact subgroups of 
mapping class groups  and $\Out(F_n)$ \cite{DM}. 
For this, we use the treatment of stable subgroups initiated by Cordes--Durham \cite{cordesdurham16}, which identifies subgroup stability with a form of convex cocompactness. See Section \ref{sec:coco} for details.
The other two properties in Theorem \ref{th:intro_main} require an essentially different argument, and for this we follow ideas from Hruska--Wise \cite{hruska-wise}. In fact, we establish a somewhat stronger property than bounded packing: For a finite collection of pairwise close stable subsets, there exists a coarse barycenter, or equivalently, stable subsets satisfy a "coarse Helly property" (see Proposition \ref{cbc} for a precise statement).

Combining Theorem \ref{th:intro_main} with existing work of a number of authors, we have the following
(see \cite{cordesdurham16} for an elaboration on these examples):

\begin{cor}\label{cor:egs}
	Let the pair $(G,H)$ of a finitely generated group and a subgroup $H$ satisfy one of the following:
	
	\begin{enumerate}
		\item[(1)] $G$ is hyperbolic and $H$ is quasiconvex in $G$.
	\item[(2)] $G$ is relatively hyperbolic and $H$ is a finitely generated subgroup quasi-isometrically embedded in
	the coned off graph \cite{farb-relhyp}.
	\item[(3)] $G = A(\Gamma)$ is a RAAG with $\Gamma$ a finite  graph which is not a join and $H$ is 
	a finitely generated subgroup quasi-isometrically embedded in the extension graph. 
	\item[(4)] $G = \Mod(S)$ and $H$ is a convex cocompact subgroup.
	\item[(5)] $G = \Out (F_n)$ and $H$ is a convex cocompact subgroup.
	\item[(6)] $H$ is an hyperbolic, hyperbolically embedded subgroup of $G$.
	\end{enumerate}
Then $H$ has finite height, finite width and bounded packing.
\end{cor}

\begin{proof}
Item $(1)$ is the content of \cite{GMRS} and \cite{hruska-wise}. The other items follow from Theorem \ref{th:intro_main} along with work by various authors establishing that $H$ is stable in $G$. For $(2)$ this is due to \cite{ADT}, for $(3)$ this is due to \cite{kms}, for $(4)$ this is due to \cite{DT15},  for $(5)$ this is again \cite{ADT}, and for $(6)$ this is in \cite{SistoZ} (see Remark \ref{HEareStable}).
\end{proof}

We also give a couple of straightforward applications.  Using work in
\cite{sageev1}, \cite{sageev2}, \cite{hruska-wise}, we have the following.
\begin{cor} (See Corollary \ref{fincubcor}.)	Suppose
	$	H$
	is a finitely generated stable codimension $1$
	subgroup
	of a finitely generated group
	$G$.  Then the corresponding
	CAT$(0)$
	cube complex
	is finite dimensional.
\end{cor}

Similarly, using work in
\cite{schwarz-inv}, \cite{mahan-relrig}, we prove the following:

\begin{prop} (See Proposition \ref{relrigprop}.)
	Let $G_1, G_2$ be finitely generated groups
	with Cayley graphs $\Gamma_1, \Gamma_2$, word metrics $d_1, d_2$,
	and stable subgroups $H_1, H_2$. Let $\Lambda_1$, $\Lambda_2$ be the limit sets of $H_1, H_2$ in 
	the Morse boundaries $\partial_M G_1, \partial_M G_2$ respectively. Let $\JJ_i$,
	$i=1,2$, be the collection  of translates of $H_{wi}$, the weak hulls
	of $\Lambda_i$ in $\Gamma_i$. Let
	$\phi: \JJ_1 \to \JJ_2$ be uniformly proper.
	Then there exists a quasi-isometry $q$ from $\Gamma_1$ to $\Gamma_2$ which pairs the sets $\JJ_1$ and $\JJ_2$ as $\phi$ does (in particular, $\Gamma_1, \Gamma_2$ are quasi-isometric).
\end{prop}

\subsection*{Applications to bounded cohomology}
We study the case of hyperbolically embedded subgroups in more detail, with the goal of obtaining the applications to bounded cohomology explained below.

Hyperbolically embedded subgroups were introduced by Dahmani, Guirardel and Osin \cite{DGO} as a generalization of parabolic subgroups of relatively hyperbolic group. They are very useful to study acylindrically hyperbolic groups, and in fact they are essential, for example, in the proof of SQ-universality \cite{DGO}, as well as in the study of bounded cohomology \cite{HullOsin,FPS, HS}. In the specific case when $G$ is hyperbolic, a subgroup $H$ is hyperbolically embedded if and only if $H$ is quasiconvex and almost malnormal. This extra structure allows to obtain stronger structural properties for the intersections. 
For example, in Proposition \ref{thm:main}, we show that if $H_1$ and $H_2$ are hyperbolically embedded in $(G,X)$, with the Cayley graph $\ga(G,X)$ hyperbolic, then the collection of essentially distinct intersections of $H_1$ with $G$-conjugates of $H_2$, is hyperbolically embedded in $(G,X)$. 

% \begin{prop}(See Proposition \ref{thm:main})
% 	Let $G$ be a group with generating set $X$ and assume that $\ga(G,X)$ is hyperbolic. Let $H_1$ and $H_2$ be hyperbolically embedded in $(G,X)$. Then
% 	\begin{enumerate}
% 		\item[{\rm (i)}] there are finitely many $H_1$-conjugacy classes of subgroups of the form $H_1\cap H_2^g$ with $g\in G$.
% 		\item[{\rm (ii)}]  If $K_1,\dots, K_n$ are representatives of the $H_1$-conjugacy classes of $H_1\cap H_2^g$, then $H_1$ is hyperbolic relative to $K_1,\dots,K_n$.
% 		\item[{\rm (iii)}] the collection $\{K_1,\dots,K_n\}$ is hyperbolically embedded in $(G,X)$.
% 	\end{enumerate}
% \end{prop}

Our main application of intersections of hyperbolically embedded subgroups is related to the bounded cohomology of acylindrically hyperbolic groups via quasimorphisms. We briefly recall this relation, see \cite{Fri} for details.

There are  natural maps from $H^{\bullet}_b(G, \mathbb{R})$, the {\bf bounded cohomology groups} of $G$ with trivial coefficients, to $H^{\bullet}(G, \mathbb{R})$, the standard cohomology groups of $G$. The kernel of theses maps is denoted by $EH^{\bullet}_b(G, \mathbb{R})$, and it is called the {\bf exact bounded cohomology} of $G$ with trivial coefficients. Thus, the problem of understanding the bounded cohomology of a group can be reduced to the problem of understanding the exact bounded cohomology.  In degree 2, there is an explicit description of $EH^{2}_b(G, \mathbb{R})$ in terms of quasimorphisms.

Recall that a {\bf quasimorphism} of $G$ is a function $f\colon G\to \mathbb{R}$ satisfying the property that there is some $D\geq 0$ such that $|f(xy)-f(x)-f(y)|\leq D$ for all $x,y\in G$. A quasimorphism is {\bf homogeneous} if $f(x^n)=n f(x)$ for all $n\in \mathbb{N}$ and all $x\in G$. Denote by $Q^h(G,\mathbb{R})$ the space of homogeneous quasimorphisms. One has that 
$$Q^h(G,\mathbb{R})/Hom(G,\mathbb{R})\cong EH_b^2(G,\mathbb{R}).$$

Brooks \cite{Brooks} develop a combinatorial construction that gives an infinite dimensional family of  quasimorphims for free groups. This construction can be viewed as extending a quasimorphism from an infinite cyclic subgroup to the whole group. There have been many generalizations of this construction, see e.g. \cite{EpsteinFuji, Fujiwara1,BeFu-wpd, FujiTAMS}, 
the most general of which being in the context of hyperbolically embedded subgroups (see \cite{HullOsin,FPS}).
%and the most general context where it has been applied is for hyperbolically embedded subgroups (see \cite{HullOsin,FPS}). 
In this paper, we improve this construction to extend quasimorphisms from different hyperbolically embedded subgroups simultaneously. 

Let $H_1,\dots,H_l$ be subgroups of a group $G$. Let $q_i$ be a homogeneous quasimorphism on $H_i$. We say that $\{q_i\}$ is {\bf intersection-compatible} if whenever $x\in H_i$ is conjugate to $y\in H_j$ we have $q_i(x)=q_j(y)$. It is very easy to see that if the $q_i$ have a common extension, then $\{q_i\}$ is intersection-compatible.

 \begin{theorem}(See Theorem \ref{compatible_extension}).
  Let $G$ be a group with generating set $X$ and assume that $\Gamma(G,X)$ is hyperbolic. Let $H_1,\dots,H_l$ be hyperbolically embedded in $(G,X)$, and let $q_i$ be a homogeneous quasimorphism on $H_i$.
  Then there exists a homogeneous quasimorphism $q$ on $G$ so that $q|_{H_i}=q_i$ if and only if $\{q_i\}$ is intersection-compatible.
 \end{theorem}

As an application of this, we show that not all the exact bounded cohomology of $G$ in degree 2 can be recovered from a given finite family of hyperbolically embedded subgroups (Corollary \ref{EHbnot_injective}), despite the fact that it can be recovered from the collection of all virtually free hyperbolically embedded subgroups by \cite{HS}.

\section{Preliminaries}
\subsection{Subgroup stability} \label{sec:stab}
We recall some material from \cite{DT15, cordes15}.
\begin{defn}
 Let $X$ be a geodesic metric space and $f \colon \R_{\geq 1}\times \R_{\geq 0} \rightarrow \R_{\geq 0}$ be
	a function.
	A quasigeodesic $\gamma$ in $X$ is  {\bf $f$--stable} if  for any $(K,\epsilon)$-quasigeodesic $\eta$ with endpoints on 
	$\gamma$,  we have $\eta \subset \ N_{f(K,\epsilon)}(\gamma)$, the $f(K,\epsilon)$-neighborhood of $\gamma$.
	
	 The function $f$ is called the {\bf stability function} of $\gamma$.
\end{defn}

\begin{defn} 
	Let $\Phi \colon X \rightarrow Y$ be a quasi-isometric embedding between geodesic metric spaces.  $\Phi(X)$ is called a {\bf stable subspace} of $Y$ and $\Phi$ is called a {\bf stable embedding} if there exists a stability function $f$ such that the image of any geodesic in $X$ is an $f$-stable quasigeodesic in $Y$.
	
	If $H\le G$ are finitely generated groups,  $H$ is called a {\bf stable subgroup} of $G$ if the inclusion map $i\colon H \to G$ is a stable embedding.
\end{defn}

Note that if $H$ is a stable subgroup of $G$, then the coset $gH$ is also stable in $G$ with the same stability function. We will need the following basic lemma, which follows from the definitions.

\begin{lemma} \label{lem:basic}
Suppose that $H$ is a stable subgroup of a finitely generated group $G$. For any $D\ge0$ there exists $R\ge0$ such that if $\gamma$ is a geodesic in $G$ whose endpoints have distance $D$ from $H$ then
\[
\gamma \subset N_R(H).
\] 
Moreover, such a geodesic is stable with stability function depending only on $D$ and that of $H$.
\end{lemma}

We conclude this section by recalling a result of Cordes:

\begin{lemma}[{\cite[Lemma 2.2 and 2.3]{cordes15}}]\label{lem:cordes}
Suppose that $X$ is a geodesic metric space and let $\alpha$, $\beta$, and $\gamma$ be the sides of a geodesic triangle $\Delta$ in $X$. If $\alpha$ and $\beta$ are $f$--stable, then $\Delta$ is $\delta$--thin and $\gamma$ is $f'$--stable, where $\delta \ge0$ and $f'$
depend only on $f$.
\end{lemma}

\subsection{Boundary convex cocompactness} \label{sec:coco}
The {\bf Morse boundary} $\partial_M X$
 of a geodesic metric space $X$ was defined in \cite{cordes15, cordeshume16}.
 We shall not need
the exact definition. Suffice to say that one can think of it 
as consisting of asymptote classes of sequences of points which can be connected to a fixed basepoint $o \in X$ by stable geodesics.  
Let $H$ be a finitely generated group acting by isometries on a proper geodesic metric space $X$.   The {\bf limit set} $\Lambda_H \subset \partial_M X$ is the set of points which can be represented by sequences in $H \cdot o$. 
It is proved in \cite{cordesdurham16} that $\Lambda_H$ is independent of the base-point $o$.
The {\bf weak hull} $H_w(H)$ of $\Lambda_H$ in $X$ is the union of  all geodesics with distinct endpoints in $\Lambda_H$. 

\begin{defn}\cite{cordesdurham16}
 $H$ acts boundary convex cocompactly on $X$ if
	\begin{enumerate}
		\item $H$ acts properly on $X$;
		\item For some (any) $o \in X$, $\Lambda_H$ is nonempty and compact; 
		\item For some (any) $o \in X$, the action of $H$ on $H_w(H)$ is cocompact.
	\end{enumerate} 

\noindent Let $G$ be a finitely generated group and $H \subset G$ a subgroup.  $H$ is {\bf boundary convex cocompact}
in $G$ if $H$ acts boundary convex cocompactly on any Cayley graph of $G$ with respect to a finite generating set. 
\end{defn}

The following theorem was proven by Cordes and Durham:
\begin{theorem}\cite{cordesdurham16}\label{coco=stable}
	Let $G$ be a finitely generated group and $H$ a subgroup. Then $H$ is stable in $G$ if and only if it is boundary convex cocompact. 
	
	In this case,  $H$ is hyperbolic and the inclusion map $i: H \to G$  extends continuously and $H$-equivariantly to an embedding of the Gromov boundary $\partial H$ into $\partial_M G$ which is a homeomorphism onto its limit set
	 $\Lambda_H$.
\end{theorem}

\subsection{Height, Width and Bounded Packing}
	For a subgroup $H \le G$, we write $H^g = gHg^{-1}$.
\begin{defn}\cite{GMRS}
	Let $G$ be a group and $H$ a subgroup.
	\begin{itemize}
 	\item Conjugates $H^{g_1}, \cdots, H^{g_k}$ are {\bf essentially distinct}
	if the cosets $g_1H, \cdots, g_kH$ are distinct.
	
	\item $H$ has height at most $n$ in $G$ if the intersection of any $(n+1)$ essentially distinct conjugates is finite. The least $n$ for which this is satisfied is called the {\bf height} of $H$ in $G$.
	
	\item The {\bf width} of $H$  is the maximal cardinality of the set $\{g_iH : |H^{g_i} \cap H^{g_j}| = \infty \}$, where $\{g_iH\}$ ranges over all collections of distinct cosets.
	\end{itemize}
Similarly,  given  a finite collection $\HH =\{H_1, \cdots, H_l\}$, 
conjugates $H_{\sigma(1)}^{g_1}, \cdots, H_{\sigma(k)}^{g_k}$ are {\bf essentially distinct}
 if the cosets $g_1H_{\sigma(1)}, \cdots, g_kH_{\sigma(k)}$ are distinct. 
% Here, $\sigma \colon \{1, \ldots, k\} \to \{1, \ldots, l\}$.
 The finite collection $\HH$ of subgroups of $G$ has height at most $n$ if the intersection of any $(n+1)$ essentially distinct conjugates is finite.
	 The width of $H$  is the maximal cardinality of the set $\{g_{\sigma(i)}H : |H_{\sigma(i)}^{g_i} \cap H_{\sigma(j)}^{g_j}| = \infty \}$, where $\{g_{\sigma(i)}H\}$ ranges over all collections of distinct cosets.
\end{defn}
Note that height (resp. width) of a finite family $\HH$ is zero if and only if all the subgroups of $\HH$ are finite.

A geometric analog of height was defined by Hruska and Wise.

\begin{defn}\cite{hruska-wise}
	Let $G$ be a  finitely generated group and $\Gamma$ a Cayley graph with respect to a 
finite generating set. A subgroup $H$ has {\bf bounded packing} in $G$  if, for all $D \geq 0$,
there exist  $N\in \natls$ such that for any collection of $N$ distinct cosets
$gH$ in $G$, at least two are separated by a distance of at least   $D$.

Similarly,    a finite collection $\HH =\{H_1, \cdots, H_l\}$ of 
 subgroups of $G$ has {\bf bounded packing}  if, for all $D \geq 0$,
 there exist  $N\in \natls$ such that for any collection of $N$ distinct cosets
 $gH_i: H_i \in \HH$, at least two are separated by a distance of at least   $D$.
\end{defn}

Equivalently, $\HH$ has bounded packing in $G$ if for each $D$ there exists $N = N(D)$ bounding the cardinality of collections $\{g_1H_{\sigma(1)},\cdots, g_rH_{\sigma(r)}\}$ of pairwise $D-$close distinct left cosets
of elements of $\HH$. Here, closeness is with respect to minimum distance between cosets.

\subsection{Hyperbolically embedded subgroups}
Hyperbolically embedded subgroups generalize the concept of parabolic subgroups of relatively hyperbolic groups (Lemma \ref{lem:relhyp}) and every hyperbolic hyperbolically embedded subgroup is stable (Remark \ref{HEareStable}).
%Let us recall some definitions and statements.
%  \begin{defn}[{\cite[Def. 4.25]{DGO}}]\label{def:he}
% Let $G$ be a group and let $\Hl $ be a family of subgroups of $G$. Suppose that $X$ is a {\it relative generating set} of $G$ with respect to $\Hl$
% (i.e., $G=\gen{X\cup \bigcup_{\lambda\in \Lambda} H_\lambda}$). Note that $X$ could be infinite.
% Denote
% \begin{equation}\label{eq:cH}
% \mathcal{H}= \bigsqcup_{\lambda \in \Lambda} (H_\lambda\setminus \{1\}).
% \end{equation}
% 
% 
% The family $\{H_\lambda\}_{\lambda \in \Lambda}$ is {\bf hyperbolically embedded} in $G$ with respect to $X$ (notation: $\{H_\lambda\}_{\lambda \in \Lambda} \h (G,X)$),
% if the Cayley graph $\G$ is hyperbolic and 
% for every $\lambda \in \Lambda$ and every $n\in \mathbb{N}$, there are finitely many geodesic paths of length at most $n$ in $\G-\ga_\lambda$ with end points in $1$ and $h\in H_\lambda$.
% Here $\G-\ga_\lambda$ denotes the graph $\G$ where we have removed the edges  of the copy $\ga(H_\lambda, H_\lambda)$ in $\G$.
% 
% We write $\{H_\lambda\}_{\lambda \in \Lambda}\h G$ if there is some set $X$ such that $\{H_\lambda\}_{\lambda \in \Lambda}\h (G,X)$.
% \end{defn}
% 
% We note that if $H$ is finite or $H=G$, then $H\h G$. These are called the {\bf degenerate cases}. Observe that for degenerate hyperbolically embedded subgroups Theorem \ref{th:intro_main} trivially holds.  

Recall also that, by a theorem of Osin \cite{Osin}, a group is {\bf acylindrically hyperbolic} if it contains a non-degenerate hyperbolically embedded subgroup.

We do not give the general definition of hyperbolically embedded subgroup, and we refer the reader to \cite[Def. 4.25]{DGO}. Instead, we restrict to the case of hyperbolically embedded subgroup which are hyperbolic, the one of interest in this paper, and give the characterization that will be most useful for our purposes. The following was proved in \cite[Theorem 3.9]{AMS} and
independently proved by Hull in \cite[Thm. 4.13]{Hull}.

\begin{theorem}\label{thm:addQ}
Suppose that $G$ is a group and $X$ is a possibly infinite generating set so that $\cay{X}{G}$ is hyperbolic.
% , $\Hl$  is a collection of subgroups of $G$ and $X$ is a relative generating set of $G$ with respect to $\Hl$, such that  $\Hl \h (G,X)$.
% Set $X_1:=X \sqcup \cH$ where $\cH$ is as in \eqref{eq:cH}.

A family $\{Q_i\}_{i=1}^n$  of subgroups of $G$ is hyperbolically embedded in $(G,X)$, denoted $\{Q_i\}_{i=1}^n\h (G,X)$, if and only if the following hold: 
\begin{enumerate}

\item[{\rm (Q1)}] ({\it $\{Q_i,\}_{i=1}^n$ is geometrically separated}) For every $\e>0$ there exists $R=R(\e)$ such that for $g\in G$ if $$\diam (Q_i \cap \nei{gQ_j}{\e}) \ge R$$
then $i=j$ and $g\in  Q_i$ (here the distances are measured with respect to the graph metric on $\cay{X}{G}$).

\item[{\rm (Q2)}] ({\it Finite generation}) For each $i$,  there exists a finite subset $Y_i\subset G$   generating  $Q_i$.
\item[{\rm (Q3)}] ({\it Quasi-isometrically embedded}) There exist $\mu\geq 1$ and $c\geq 0$ such that for any $i \in \{1,\dots,n\}$ and all $h\in Q_i$ one has $|h|_{Y_i}\leq \mu |h|_{X}+c$.
\end{enumerate}
  \end{theorem}
  
  We write $\{Q_i\}_{i=1}^n\h G$ if there is some set $X$ such that $\{Q_i\}_{i=1}^n\h (G,X)$.
  
% \begin{rem}\label{rem:addQtrivial}
% We note that the theorem applies even in the case where 
% $\Hl=\{1\}$ is the empty family, in which case the only requirement is that $(G,X)$ is hyperbolic. Therefore the theorem can be viewed as a definition for a family to be hyperbolically embedded in $(G,X)$ in the case that $(G,X)$ is hyperbolic.
% \end{rem}

Hyperbolically embedded subgroups generalize peripheral subgroups of relatively hyperbolic groups in the following sense:

\begin{lemma}\cite[Proposition 4.28]{DGO}\label{lem:relhyp}
A group $G$ is relatively hyperbolic with respect to $\Hl$ if and only if for any finite relative generating set $X$, $\Hl\h (G,X)$.
\end{lemma}

\begin{rem}\label{HEareStable}
It was proved in \cite[Theorem 2]{SistoZ} that if $H$ is a finitely generated subgroup of $G$ such that $H\h G$, then $H$ is quasiconvex. That is, for  every $\mu\geq 1$ and $c\geq 0$, there exists $C$ such that all $(\mu,c)$-quasigeodesics in $G$ whose  endpoints are in $H$ lie in $N_C(H)$.
So, $H$ is stable if (and only if) $H$ is hyperbolic.

On most of our statements we will assume that $H\h (G,X)$ where $\ga(G,X)$ is hyperbolic. It follows from  Theorem \ref{thm:addQ} that in this case,  $H$ is hyperbolic and, in particular, stable in $G$.
\end{rem}

For later purposes we record a property of cosets of hyperbolically embedded subgroups, encoded in the following definition.

\begin{defn}\label{def:WPD}
Let $G$ be a group and $X$ a generating set of $G$.
A subset $H$ of $G$ satisfies the {\it WPD condition} in $(G,X)$ if 
for every $\e>0$, there is $N>0$ such that for 
every $h_1,h_2\in H$ satisfying $\d_X(h_1,h_2)>N$, we have
$$|\{ g\in G \mid \d_X(h_1, gh_1)< \e, \, \d_X(h_2,gh_2)<\e \}|<\infty. $$
\end{defn}

The following is proved in \cite[Lemma 3.2]{SistoZ}, and in the case where $(G,X)$ is hyperbolic, it can be deduced easily from Theorem \ref{thm:addQ}.

\begin{lemma}\label{lem:HEareWPD}
Suppose that $\Hl \h (G,X)$. Then each $H_\lambda$ is WPD in $(G,X)$.
\end{lemma}

\section{Limit Sets and Height}
In this section, we show that 
some standard arguments
in the theory of Kleinian groups go through in the more general context of stable subgroups.
A proof of the following lemma is sketched in \cite{gromov-ai} for quasiconvex subgroups of hyperbolic groups. 

\begin{lemma}[Intersection of limit sets]\label{int}
	Let $H_1, H_2$ be stable subgroups of a finitely generated group $G$ with limit sets $\Lambda_1, \Lambda_2$ in $\partial_MG$.
	Let $\Lambda_3$ be the limit set of $H_1 \cap H_2$.
	Then 
	\begin{enumerate}
		\item $H_1 \cap H_2$ is stable, and
%		and its stability function depends only on
%		those for $H_1$ and  $H_2$.
		\item $\Lambda_3 = \Lambda_1 \cap \Lambda_2$.
	\end{enumerate}   
\end{lemma}

\begin{proof} 
	Let $\Gamma$ be  a Cayley graph of $G$ with respect to a finite generating set.
	The first statement follows immediately from a result of Short \cite{short} which states that the intersections of quasiconvex subgroups of an arbitrary finitely generated group is quasiconvex.
	
	Since $H_1 \cap H_2 \subset H_1$, $\Lambda_3 \subset \Lambda_1$. Similarly, $\Lambda_3 \subset \Lambda_2$.
	Hence $\Lambda_3 \subset \Lambda_1 \cap \Lambda_2$.
	
	To prove the opposite inclusion, let $p \in \Lambda_1 \cap \Lambda_2$.
	We want to show that $p\in \Lambda_3$. Since $H_1, H_2$ are stable, there exist (uniform)
	quasigeodesics $\gamma_i \subset H_i$,
	$i=1,2$ converging to $p$. Without loss of generality, we may assume that $\gamma_1, \gamma_2$
	both start at $1 \in \Gamma$. By stability, $\gamma_1, \gamma_2$ lie in a uniformly bounded neighborhood of each other
	in $\Gamma$ \cite[Corollary 2.6]{cordes15}. In particular, there exists $g \in G$, and infinitely many $g_i \in \gamma_1$, $h_i \in \gamma_2$, 
	such that $g_ig = h_i$ for all $i$. Hence $g = g_i^{-1}h_i$ for all $i$ and so $g_ig_1^{-1} = h_i h_1^{-1} \in H_1 \cap H_2$ for all $i$, thus producing an infinite sequence of elements in $H_1 \cap H_2$ converging to $p$.
\end{proof}

The proof of the following proposition is a reprise of the elementary argument in \linebreak \cite{GMRS} using limit sets in the Morse boundary.

\begin{prop}[Finite height] \label{finht} 
Let $\HH =\{H_1, \cdots, H_l\}$ be a finite collection of 
 stable subgroups of a finitely generated group $G$. Then $\HH$ has finite height.
\end{prop}

\begin{proof} 
As before, let $\Gamma$ be the Cayley graph of $G$ with respect to a finite generating set.
Let $\Lambda_i$ denote the limit set of  $H_i$.
	By Lemma \ref{int}, the limit set of $\cap_j H_{\sigma(j)}^{g_j}$ is the intersection of
	the limit sets of $ H_{\sigma(j)}^{g_j}$, which in turn is the 
	intersection $\Lambda_\cap$ of
	the limit sets $ g_j\Lambda_{\sigma(j)}$
	of $ g_jH_{\sigma(j)}$.  
Supposing that $\cap_j H_{\sigma(j)}^{g_j}$ is infinite, $\Lambda_\cap$ consists of at least two points.

If $\gamma$ is a geodesic in the weak hull of $\Lambda_\cap$, then $g_j^{-1}\gamma$ is in the weak hull of $H_{\sigma(j)}$ and so by Theorem \ref{coco=stable}, $g_j^{-1} \gamma \subset N_D(H_{\sigma(j)})$, for some $D$ depending only on $\HH$. Hence, if $x$ is any point along $\gamma$ then $N_D(x)$ meets each $g_jH_{\sigma(j)}$.
The number of cosets of any $H_i$ is therefore bounded by the cardinality of the ball $N_D(x)$, which is finite as $G$ is finitely generated and the Cayley graph $\Gamma$ is taken with respect to a finite generating set.
The conclusion follows.
\end{proof}

Also, Theorem \ref{coco=stable} immediately gives the following:

\begin{prop}\label{comm}
Let $H \subset G$ be stable. Then $H$ is of finite index in the stabilizer of $\Lambda_H$
and hence in the commensurator $Comm_G(H)$ of $H$ in $G$.
\end{prop}

\begin{proof} The stabilizer
	$Stab (\Lambda_H)$ stabilizes the weak hull $H_w(H)$ and acts properly on it. Since $H$ acts cocompactly on
	 $H_w(H)$ by Theorem \ref{coco=stable} and $H \subset Stab (\Lambda_H)$, it follows that if $K$
	 is a compact (weak) fundamental domain for $H$ in $H_w(H)$ (i.e.\ a compact subset of  $H_w(H)$,
	 whose $H-$translates cover  $H_w(H)$) there are at most finitely
	 many elements $g$ of $Stab (\Lambda_H)$ such that $gK\cap K \neq \emptyset$. It follows that
	 $[Stab (\Lambda_H):H] < \infty$. Since $Comm_G(H) \subset Stab (\Lambda_H)$, the last assertion follows.
\end{proof}

\section{Width and Packing} \label{sec:wp}
To prove finiteness of width and bounded packing, we cannot reprise the argument
in  \cite{GMRS} which uses compactness of the boundary
and global hyperbolicity in an 
essential way. In \cite{hruska-wise}, Hruska and Wise give a new proof of finiteness of width for quasiconvex subgroups of hyperbolic groups. Here, we adapt  their argument to stable subgroups of arbitrary finitely generated groups.

\subsection{Finiteness of width and bounded packing}\label{sec:proof1}

We begin by proving bounded packing for a single stable subgroup.
\begin{theorem} \label{th:bounded_packing}
A stable subgroup $H$ of a finitely generated group $G$ has bounded packing. 
\end{theorem}

\begin{proof}
Let $\mathcal{C}$ be a collection of left coset of $H$ in $G$ whose $D$-neighborhoods intersect pairwise. Our goal is to bound the size of $\mathcal{C}$ in terms of $D$, and following \cite{hruska-wise}, we do so by induction on the height of $H$. This is possible by Proposition \ref{finht}. Recall that, for the base of induction (height  zero), $H$ is finite and trivially has bounded packing.

Translating by $G$, we may assume that $H \in \mathcal{C}$. Moreover, if $g H \in \mathcal{C}$, then $gH = hxH$ for $h \in H$ and $x \in B_D(1)$, since $gH$ is $D-$close to $H$. We fix $x \in B_D(1)$ and bound the number of cosets of the form $hxH$ in $\mathcal{C}$. Since the number of such $x$ is bounded, this will establish the theorem.

Hruska and Wise show that for any $L\ge0$ there is an $L' \ge0$ so that 
%for $x \in B_D(1)$
\begin{align} \label{near_conj}
N_L(H) \cap N_L(xH) \subset N_{L'}(K),
\end{align}
where $K=H \cap xHx^{-1}$ \cite[Lemma 4.5]{hruska-wise}. Moreover, there is a bijection, $hxH \leftrightarrow hK$, and the height of $K$ in $H$ is strictly less than the height of $H$ in $G$ \cite[Lemma 4.2]{hruska-wise}. Since $K$ is a stable subgroup of $H$ by Lemma \ref{int}, the induction hypothesis implies that $K$ has bounded packing in $H$. 

We finally claim that there is a $D'$ depending only on $D$ such that the $D'$-neighborhoods of the $hK$ intersect pairwise in $H$. By bounded packing, this will imply that the number of such cosets is bounded. Using the bijection above, this in turn bounds the size of $\mathcal{C}$ as required. To prove the claim, let $h_1, h_2$ be such that $h_1xH,h_2xH \in \mathcal{C}$. Pick points $a,b,c \in G$ such that 
\[
a \in N_D(H) \cap N_D(h_1xH), \\
b \in N_D(H) \cap N_D(h_2xH), \\
c \in N_D(h_1xH) \cap N_D(h_2xH). 
\]

Applying Lemma \ref{lem:basic}, we see that for $R\ge0$ depending on $D$,
$ [a,b] \subset N_R(H)$, $[b,c] \subset  N_R(h_2xH)$, and  $[a,c] \in N_R(h_1xH)$. As these form a triangle of uniformly stable geodesics (Lemma \ref{lem:basic}), there is a $\delta \ge0$ depending only on $D$ (and the stability function of $H$) such that this triangle is $\delta$-thin (Lemma \ref{lem:cordes}).  
Hence, there is a $w\in [a,b]$ within distance $\delta$ from both $[b,c]$ and $[a,c]$. We conclude that $w$ is contained in the intersection of the $(2R+\delta)$--neighborhoods of $H$ and $h_ixH$ for $i = 1,2$. Applying the containment  Eq. \ref{near_conj}, there is a $L'\ge0$ such that 
\[
w \in N_{2R+\delta}(H) \cap N_{2R+\delta}(h_ixH) \subset N_{L'}(h_iK)
\]
for $i =1,2$. We conclude that the $G$-distance between $h_1K$ and $h_2K$ is at most $2L'$. Since $H$ is undistorted in $G$, we conclude that there distance in $H$ is at most $D'$, which depending only on $D$ and the stability function of $H$. This concludes the proof.
\end{proof}

We now want to use bounded packing to study intersections of conjugates rather than cosets, and to pass from cosets to conjugates we will need the following lemma.

\begin{lemma}\label{close_cosets}
 Let $H_1,H_2$ be stable subgroups. Then there exists $D\geq 0$ so that whenever $|H_1^{g_1}\cap H_2^{g_2}|=\infty$ for some $g_1,g_2 \in G$, then the cosets $g_1H_1$ and $g_2H_2$ have intersecting $D$--neighborhoods.
\end{lemma}

\begin{proof}
 Set $K = H_1^{g_1} \cap H_2^{g_2}$ and let $k \in K$ be an infinite order element (notice that $K$ is hyperbolic since it is stable by Lemma \ref{int}.(1), and in particular $K$ is infinite if and only if it contains an infinite order element). Since $k$ is a Morse element of $G$, it produces two points $k^\infty$ and  $k^{-\infty}$ in 
\[
\Lambda_K = g_1 \Lambda_{H_1} \cap g_2 \Lambda_{H_2} \subset \partial_M G,
\]
using Lemma \ref{int}. Let $\gamma$ be a geodesic in the weak hull of $K$ joining $k^{-\infty}$ and $k^{\infty}$. Then $g_1^{-1}\gamma$ lies in the weak hull of $H_1$ and $g_2^{-1}\gamma$ lies in the weak hull of $H_2$. Since each $H_i$, $i=1,2$, acts cocompactly on its weak hull by Theorem \ref{coco=stable}, there is a $D$ depending only on $H_1,H_2$, such that $\gamma$ lies in the intersection of $N_D (g_1H_1)$ and $N_D(g_2H_2)$.
\end{proof}

\begin{prop}\label{conj_rep}
 Let $H_1,\dots,H_l$ be stable subgroups of a finitely generated group $G$. Then there exist finitely many conjugacy classes of subgroups of $G$ of the form $H_{i(1)}^{g_1}\cap\dots\cap H_{i(n)}^{g_n}$, where $i(j)\in\{1,\dots,n\}$ and $g_j\in G$ and $n\in\mathbb N$.
\end{prop}

\begin{proof}
 Since each $H_i$ is hyperbolic, it contains only finitely many conjugacy classes of finite subgroups, and hence we can can restrict to considering infinite intersections. We show by induction on $n$ that, whenever $K_1,\dots,K_{l'}$ are stable subgroups of $G$, there are only finitely many  conjugacy classes of infinite subgroups of $G$ of the form $K_{i(1)}^{g_1}\cap\dots\cap K_{i(n)}^{g_n}$. This suffices in view of finiteness of height (Proposition \ref{finht}).
 
For $n=1$ this is obvious. Let us prove it for $n=2$. It suffices to show that there are only finitely many conjugacy classes of infinite subgroups of $G$ of the form $K_{1}\cap K_{2}^{g}$. By the Lemma \ref{close_cosets} there is some $D\geq 0$ depending only on $K_1,K_2$  so that $gK_2$ intersects $N_D(K_1)$, so that up to multiplying $g$ on the left by an element of $K_1$ we have that $gK_2$ has a representative in the ball of radius $D$ around the identity, which is finite.

Assuming that the claim holds for some $n$, we can replace $K_1,\dots,K_{l'}$ with a set $\mathcal K$ of representatives of the conjugacy classes of infinite subgroups of the form $K_{i(1)}^{g_1}\cap\dots\cap K_{i(n)}^{g_n}$. This is a set of stable subgroups (by Lemma \ref{int}), and it is finite by the inductive hypothesis. As any intersection of $n+1$ conjugates of the $K_i$ is an intersection of two conjugates of elements of $\mathcal K$, this completes the proof.
\end{proof}

\begin{theorem}\label{widthpack}
Let $\HH =\{H_1, \cdots, H_l\}$ be a finite collection of stable subgroups of a finitely generated group $G$. Then $\HH$ has bounded packing in $G$. Further, $\HH$ has finite width.
\end{theorem}

\begin{proof}
First, let $\mathcal{C}$ be a collection of coset of subgroups in $\HH$ whose $D$--neighborhoods intersect pairwise. By Theorem \ref{th:bounded_packing}, for each $i = 1, \ldots, l$ there is a $N_i \ge0$ such that the number of coset of $H_i$ in $\mathcal{C}$ is at most $N_i$. Hence, $\# \mathcal{C} \le \sum_{i=1}^{l}N_i$ and so $\HH$ has bounded packing.

Finite width comes from the fact that, by Lemma \ref{close_cosets}, if $H_i^{g_1} \cap H_j^{g_2}$ is infinite for $i,j \in \{1, \ldots, l\}$ and $g_1,g_2 \in G$, then the cosets $g_1H_i$ and $g_2H_j$ have intersecting $D$--neighborhoods, for $D$ depending only on $\HH$. Finite width now follows from bounded packing.
\end{proof}

 \section{Hyperbolically embedded subgroups and bounded cohomology}
 %As a particular instance of Theorem \ref{widthpack}, we have the following:

\subsection{Intersections of hyperbolically embedded subgroups}

In this subsection we prove a few results about intersections of (conjugates of) hyperbolically embedded subgroups.

 \begin{prop}\label{thm:main}
 	Let $G$ be a group with generating set $X$ and assume that $\ga(G,X)$ is hyperbolic. Let $H_1$ and $H_2$ be hyperbolically embedded in $(G,X)$. Then
 	\begin{enumerate}
 		\item[{\rm (i)}] there are finitely many $H_1$-conjugacy classes of subgroups of the form $H_1\cap H_2^g$ with $g\in G$.
 		\item[{\rm (ii)}]  If $K_1,\dots, K_n$ are representatives of the $H_1$-conjugacy classes of $H_1\cap H_2^g$, then $H_1$ is hyperbolic relative to $K_1,\dots,K_n$.
 		\item[{\rm (iii)}] $\{K_1,\dots,K_n\}\hookrightarrow_h (G,X)$.
 	\end{enumerate}
 \end{prop}
 
 \begin{proof}
  (i) follows from Proposition \ref{conj_rep}, which immediately gives finitely many $G$--conjugacy classes, and the fact that hyperbolically embedded subgroups are almost malnormal \cite[Proposition 2.10]{DGO}.
  
  Once we prove (ii), (iii) follows from (ii) and \cite[Theorem 4.35]{DGO}. Hence, let us prove (ii).
  
  By Lemma \ref{lem:relhyp}, it is enough to show that $\{K_i\}_{i=1}^n \h (H_1, Y)$ where $Y$ is some finite generating set of $H_1$. Since $H_1$ is quasi-isometrically embedded in $(G,X)$, we can work with $\dx$ rather than ${\rm d}_Y$. Notice that since  
$\ga(G,X)$ is hyperbolic, $\ga(H_1, Y)$ is hyperbolic as well. Therefore, by Theorem \ref{thm:addQ}, we need to prove that each $K_i$ is finitely generated and quasi-isometrically embedded in $(H_1,\dx)$, which follows from Lemma \ref{int} (from which one deduces that $K_i$ is a finitely generated undistorted subgroup of $H_1$). We also need to show that $\{K_i\}_{i=1}^n$ is geometrically separated in $(H_1,\dx)$. 

In order to prove geometric separation let $t_i$ be so that
$K_i= H_1\cap H_2^{t_i}$. Notice that there exists $R>0$ such that for every $i$, $K_i\subseteq \nei{t_iH_2}{R}$. We claim that for every $a,b\in H_1$ if $aK_i\neq bK_j$ then $at_iH_2\neq bt_j H_2$.
Once the claim is proved, the geometric separation of $\{K_i\}_{i=1}^n$ follows directly from the geometric separation of $H_2$, since for every $g\in H_1$ we have $gK_i\subseteq \nei{gt_iH_2}{R}$.

Suppose that $at_iH_2= bt_jH_2$, and let $h_2\in H_2$ such that $at_ih_2= bt_j$, then 
$$K_i^a=H_1\cap (H_2^{t_i})^a = H_1\cap H_2^{at_i}= H_1\cap H_2^{at_ih_2}= H_1\cap H_2^{bt_j}=K_j^b.$$
In particular, $i=j$ since by assumption distinct $K_i$'s are representatives of distinct $H_1$-conjugacy classes. 

Hence,
$$aK_i=H_1\cap bt_ih_2^{-1}t_i^{-1}H_2^{t_i}=H_1\cap bH_2^{t_i}=bK_j,$$
and the claim is proved.
%The latter follows from the fact that there exists $R$ such that every coset $gK_i$ is contained in $\nei{gH_2}{R}$, so that geometric separation of $H_2$ implies geometric separation of $\{K_i\}_{i=1}^n$. Notice that distinct cosets of some $K_i$ are contained in neighborhoods of distinct cosets of $H_2$.
 \end{proof}

 \begin{lemma}\label{lem:finitenessdoublecosets}
Let $G$ be a group, let $X$ be a generating set of $G$ so that $\ga(G,X)$ is hyperbolic, and let $H_1,H_2$ be hyperbolically embedded in $(G,X)$. Then for each $R\geq 0$ there exists $C$ such that there only finitely many $s\in G$ satisfying 
\begin{equation}\label{eq:doublecosets}
\dx(1,s)\leq R\;\text{ and }\;\diam(H_1\cap \nei{sH_2}{R})\geq C.
\end{equation}
\end{lemma}
\begin{proof}
We can enlarge $X$ by adding finitely many elements to ensure that $X\cap H_i$ is a generating set of $H_i$.

Notice that there exists $R'\geq R$, depending on $R$ and the quasiconvexity constants of $H_1,H_2$ only, so that for every $d$ if there exists $a\in H_1\cap\nei{sH_2}{R}$ satisfying $\dx(1,a)\geq d$ then there exists $a'\in H_1\cap\nei{sH_2}{R'}$ satisfying $\dx(1,a')= d$.

Fix $R>0$, and let $R'$ be as above. Since $H_2$ satisfies the WPD condition (see Definition \ref{def:WPD}), letting $\e=2R'$,  there is $N$ such that for every $b\in H_2$ satisfying that $\dx(1,b)>N$ we have that 

\begin{equation}\label{eq:H2WPD}
|\{g\in G\mid \dx(1,g)<\e, \, \dx(b,gb)<\e\}|<\infty.
\end{equation}
We put $C=N+2R'$.

Let $S\subset G$ be the set
$$S=\{ s\in G \mid \dx(1,s)\leq R, \, \diam(H_1\cap \nei{sH_2}{R})\geq C\}.$$

Assume that $S$ is infinite. 
Since the metric $\dx$ restricted to  $H_1$ is proper, by the definition of $R'$ there is an infinite subset $S_1$ of $S$ and $a\in H_1$, 
such that $a\in H_1\cap \nei{sH_2}{R'}$ for all $s\in S_1$ 
and $\dx(1,a)=C$.
For each $s_i\in S_1$, let $t_i\in H_2$ such that
 $\dx(a, s_i t_i)\leq R'$.
Note that 
$$\dx(1,t_i)=\dx(s_i, s_it_i)\geq \dx(1,a)-\dx(1, s_i ) -\dx(a, s_it_i)\geq C-2R'=N
.$$ 
Since we similarly have that $\dx(1,t_i) \le C+2R'$, and the metric $\dx$ restricted to $H_2$ is proper,
there is an infinite subset $S_2$ of $S_1$ and $b\in H_2$ such that $t_i=b$ for all $s_i\in S_2$. 

In particular, for all $s_1,s_2\in S_2$, 
$$ \dx(1,b)>N \quad \text{ and } \quad
\dx(s_1,s_2)\leq 2R\leq\e \quad \text{ and } \quad
\dx(s_1b,s_2b)\leq 2R'=\e,
$$
or equivalently,
$$ \dx(1,b)>N \quad \text{ and } \quad
\dx(1,s_1^{-1}s_2)\leq 2R\leq\e \quad \text{ and } \quad
\dx(b,s_1^{-1}s_2b)\leq 2R'=\e.
$$
By \eqref{eq:H2WPD}, i.e.  the WPD condition of $H_2$, there are only finitely many different possibilities for $s_1^{-1}s_2$. %In particular, there are distinct $s_3,s_4\in S_2$ such that $h_1^{-1}s_1^{-1}s_3h_3=h_1^{-1}s_1^{-1}s_4h_4$, and hence $s_3h_3=s_4h_4$ contradicting that $H_1s_3H_2\neq H_1s_4H_2$.
\end{proof}
 
 \begin{prop}\label{prop:algint=geoint}
Let $G$ be a group, let $X$ be a generating set of $G$ so that $\ga(G,X)$ is hyperbolic, and let $H_1,H_2$ be hyperbolically embedded in $(G,X)$. Then for every $R\geq 0$, we have $\dhau(H_1 \cap \nei{H_2}{R}, H_1\cap H_2)<\infty$.
\end{prop}
\begin{proof}
Again, we enlarge $X$ by adding finitely many elements to ensure that $X\cap H_i$ is a generating set of $H_i$.

Notice that $H_1\cap H_2 \subseteq H_1 \cap \nei{H_2}{R}$, so we only need to prove that any point in  $H_1 \cap \nei{H_2}{R}$ is at a bounded distance from a point in $H_1\cap H_2$.

Let $h_1\in H_1 \cap \nei{H_2}{R}$ and $h_2\in H_2$ such that 
$\dx(h_1,h_2)\leq R$. Let $h\in H_1\cap H_2$ such that
$\d_{H_1}(h,h_1)$ is minimal, where $\d_{H_1}$ is the path metric on $H_1$ (which is quasi-isometric to $\dx$ restricted to $H_1$). Without loss of generality, we can assume that $h=1$. Let $\gamma_i$ be geodesics in $H_i$ from $1$ to $h_i$. By quasiconvexity and hyperbolicity, there exists $R'$ such that 
$\dhau(\gamma_1,\gamma_2)< R'$ where $R'$ only depends on $R$, $\delta$, and the quasiconvexity constants of $H_1$ and $H_2$.

For every vertex $x\in \gamma_1$, let $s_x$ be so that $xs_x$ is a vertex $\gamma_2$, and $\dx(x,xs_x)\leq R'$. By Lemma \ref{lem:finitenessdoublecosets}, and the observation that $H_1 \cap N_{R'}(s_xH_2) = x^{-1}(H_1 \cap N_{R'} (H_2))$,
there is $C$ and $N$ such if $\d_{H_1}(1, h_1)>C$, then there at most $N$ possibilities for $s_x$.
If $\d_{H_1}(1,h_1)>C+N+1$, then there are distinct vertices $x,y\in \gamma_1$ such that $s_x=s_y$. We assume that $x$ is  closer than $y$ to $1$ (in the $\d_{H_1}$ metric).

Observe that $yx^{-1}\in H_1$. Observe also that $yx^{-1}\in H_2$, since $yx^{-1} = (ys_x)(s_x^{-1}x^{-1})$ .
Notice that $$\d(yx^{-1}, h_1)\leq \d_{H_1}(yx^{-1},y)+\d_{H_1}(y,h_1)=\d_{H_1}(1,x)+\d_{H_1}(y,h_1)<\d_{H_1}(1, h_1)$$ which contradicts the minimality of $\d_{H_1}(1,h_1)$. 
\end{proof}

 \subsection{Extension of quasimorphisms}
 In this subsection, we shall prove a bounded cohomology application of Proposition \ref{thm:main}. 
 %for each $g\in G$ and $i,j\in\{1,\dots,l\}$ we have
 %$$q_i|_{H_i\cap H_l^g}=\left(q_l|_{H_1^{g^{-1}}\cap H_l}\right)^g.$$
 
 Recall that, given $q_i$ a quasimorphism on $H_i\leqslant G$, $i\in I$, we say that $\{q_i\}$ intersection-compatible, if whenever $x\in H_i$ and $y\in H_j$ are $G$-conjugate one has that $q_i(x)=q_j(y)$.
 \begin{remark}\label{conj_inv}
  Notice that, since homogeneous quasimorphisms are conjugation invariant, if the $q_i$ are the restriction of some homogeneous quasimorphism on $G$ then $\{q_i\}$ is intersection-compatible.
 \end{remark}

%  \begin{prop}
%  	In the notation of Proposition \ref{thm:main}, consider the map $res: EH^n_b(G)\to EH^n_b(H_1) \oplus EH^n_b(H_2)$ given by the restriction maps to the factor. Then $(\alpha_1,\alpha_2)$ is in the image of $res$ if and only if it is intersection-compatible.
%  \end{prop}
 
 \begin{theorem}\label{compatible_extension}
  Let $G$ be a group with generating set $X$ and assume that $\Gamma(G,X)$ is hyperbolic. Let $H_1,\dots,H_l$ be hyperbolically embedded in $(G,X)$, and let $q_i$ be a homogeneous quasimorphism on $H_i$.
  Then there exists a homogeneous quasimorphism $q$ on $G$ so that $q|_{H_i}=q_i$ if and only if $\{q_i\}$ is intersection-compatible.
 \end{theorem}

 We briefly recall the construction of extensions of quasicocycles from \cite{FPS}, specializing it to quasimorphisms. Let $H\hookrightarrow_h (G,X)$. Let $\hat G$ be obtained from $\ga(G,X)$ by adding, for every coset $gH$ of $H$, a vertex $c(gH)$ connected to all elements of $gH$ by an edge of length $1/4$. For $g,x\in G$, denote by $\pi_{gH}(x)$ the set of all entrance points in $gH$ of geodesics in $\hat{G}$ from $x$ to $c(gH)$. Let $D\geq 0$ be sufficiently large. 
 
 For $g,x\in G$, define the trace $tr_{gH}(x)\in\mathbb R G$ of $x$ onto $gH$ to be $0$ if $\diam_{gH}(\pi_{gH}(\{1,x\}))\leq D$, and let $tr_{gH}(x)$ be the average of all $y^{-1}z$ for $y\in\pi_{gH}(1)$ and $z\in\pi_{gH}(x)$ otherwise. Notice that it does not depend on the choice of coset representative $g$.
 %For a tuple $\overline g=(g_0,\dots,g_k)$ one can define the trace $tr_{gH}(\overline g)$ of $\overline g$ onto $gH$ to be $0$ if $diam_{gH}(\pi_{gH}(\overline g))\leq D$, and let $tr_{gH}(\overline g)\in C_k(G)$ be the average of all tuples $(h_1,\dots,h_k)$ with $h_i\in\pi_{gH}(g_i)$ otherwise.
 
 Let $q$ be an alternating quasimorphism on $H$, meaning that $q(h)=-q(h^{-1})$ for each $h\in H$. Then one can define
 $$\Theta(q)(x)=\sum_{gH}q(tr_{gH}(x)).$$
 In \cite{FPS} it is shown that $\Theta(q)$ is (well-defined and) an alternating quasimorphism, and $\Theta(q)|_{H}$ is within bounded distance of $q$. We will denote by $\Psi(q)$ the unique homogeneous quasimorphism within bounded distance of $\Theta(q)$. In particular, if $q$ is homogeneous then $\Psi(q)|_{H}=q$. 
 
 For later purposes, we also record the following lemma about $\pi_{H}$. For $x\in G$, denote by $\rho_H(x)$ the set of all $h\in H$ that minimize the distance from $x$ in $\ga(G,X)$ (this is a uniformly bounded set by quasiconvexity). Denote by $d_H$ a word metric on $H$. We note that the $d_H$-diameter of $\pi_{gH}(x)$ is uniformly bounded as a consequence of \cite[Lemma 2.9]{FPS}.
 
 \begin{lemma}\label{proj_are_same}
  In the notation above, there exists $C$ so that, for each $x\in G$, $\pi_{H}(x)$ is within Hausdorff $d_H$-distance $C$ of $\rho_H(x)$.
 \end{lemma}

 \begin{proof} 
 We will use the fact that, by \cite[Lemma 2.8]{FPS}, there exists $B\geq 1$ so that if $w,y$ satisfy $d_H(\pi_H(w),\pi_H(y))\geq B$ then any geodesic in $\hat{G}$ from $w$ to $y$ goes through $c(H)$. Moreover, for each $g\not\in H$ we have $\diam_H(\pi_H(gH))\leq B$. Finally, we will also use the fact that any geodesic in $Y=\ga(G,X)$ is an unparametrized quasigeodesic of $\hat{G}$ with uniform constants by \cite[Proposition 2.6]{KapovichRafi}.
 
 Let $h\in\rho_h(x)$. If $d_H(h,\pi_H(x))\leq 10B$ then we are done, so suppose that this is not the case. Let $\gamma$ be a geodesic in $\ga(G,X)$ from $h$ to $x$, parametrized by arc length. Since $\diam_H(\pi_H(x))\leq B$, there exist a minimal integer $t$  so that $d_H(\pi_H(\gamma(t)),\pi_H(x))\leq B$.
  %there exists a geodesic in $\hat{G}$ from $\gamma(t)$ to some fixed $x'\in\pi_H(x)$ that does not go through $c(H)$ (notice that geodesics from $x$ to $\pi_H(x)$ do not contain $c(H)$).
  
  We now provide a uniform bound $R$ on $d_Y(\gamma(t),x')$, for some fixed $x'\in\pi_H(x)$. This in turn yields a bound on $d_Y(h,x')$ and hence $d_H(h,x')$. In fact, $d_Y(\gamma(t),H)=t$, so we must have $t\leq R$, and hence we get $d_Y(h,x')\leq 2R$.
  
  Let us bound $d_Y(\gamma(t),x')$. Let $\alpha$ be a geodesic in $\hat{G}$ from $x$ to $x'$. First of all, since $\gamma$ is an unparametrized quasigeodesic of the hyperbolic space $\hat{G}$, there exists a uniform $\delta$ so that $d_{\hat{G}}(\gamma(t), x'')\leq \delta$ for some $x''\in\alpha$. Since $d_H(\pi_H(\gamma(t-1)),\pi_H(x''))>B$ (because $\pi_H(x'')\subseteq \pi_H(x)$), we get that any geodesic from $\gamma(t-1)$ to $x''$ goes through $c(H)$ and hence $d_{\hat{G}}(x'',H)\leq \delta$. Hence $d_{\hat{G}}(x'',x')\leq \delta$, because $\alpha$ is a geodesic from $x$ to some point in $\pi_H(x)$. To sum up, we have $d_{\hat{G}}(x',\gamma(t))\leq 2\delta$.
  
  In order to conclude the proof, we now show that there does not exist any coset $gH$ with $d_{gH}(\pi_{gH}(\gamma(t)),\pi_{gH}(x'))>10B$. Once this is done, we can consider a geodesic $\beta$ in $\hat{G}$ from $\gamma(t)$ to $x'$ and replace its subsegments of length $1/2$ around the apices $c(gH)$ that $\beta$ goes through with geodesics in $gH$, thereby obtaining a path of uniformly bounded length.
  
  If we had some $gH$ with $d_{gH}(\pi_{gH}(\gamma(t)),\pi_{gH}(x'))>10B$, then we would have $d_{gH}(\pi_{gH}(\gamma(t-1)),\pi_{gH}(x'))>5B$. But then any geodesic from either $\gamma(t-1)$ or $\gamma(t)$ to $c(H)$ would contain $c(gH)$, and hence we would have $\pi_H(\gamma(t))=\pi_H(c(gH))=\pi_H(\gamma(t-1))$, contradicting the minimality of $t$.
 \end{proof}

\begin{proof}[Proof of Theorem \ref{compatible_extension}]
By Remark \ref{conj_inv}, we only need to show that we can always extend an intersection-compatible family $\{q_1,\dots,q_l\}$ of quasimorphisms. For $l=1$, we can use \cite{HullOsin} or $\Psi$ as constructed above.

  Let now $l\geq 2$ and let $H_1,\dots,H_l<G$ be subgroups and suppose $H_i\hookrightarrow_h (G,X)$. Let $\hat G^1$ be obtained from $\ga(G,X)$ by adding, for every coset $gH_1$ of $H_1$, a vertex $c(gH_1)$ connected to all elements of $gH_1$ by an edge of length $1/4$. Let $K_1,\dots, K_n$ are representatives of the $H_1$-conjugacy classes of subgroups of the form $H_1\cap H_j^g$ where $j\geq 2$ (there are finitely many conjugacy classes by Theorem \ref{thm:main}).
 
 For $g,x\in G$, denote $\pi_{gH_1}(x)$ the set of all entrance points in $gH_1$ of geodesics in $\hat{G}^{1}$ from $x$ to $c(gH)$.
 
 \begin{lemma}\label{lem:proj_is_peripheral}
 	There exists $C\geq0$ with the following property. For each $g\in G$ and $j\in\{2,\dots,l\}$ there exists $h\in G$ and $i\in\{1,\dots,n\}$ so that $hK_i\subseteq gH_1$ and so that for all $x\in H_{j}$ we have $d_{gH_1}(\pi_{gH_1}(x),hK_i)\leq C$.
 \end{lemma}
 
 \begin{proof}
 In view of Lemma \ref{proj_are_same}, we can consider the closest point projection $\rho_{gH_1}$ instead of $\pi_{gH_1}$. If the diameter of $\rho_{gH_1}(H_j)$ is sufficiently large, then, by quasiconvexity, $gH_1$ and $H_j$ are at uniformly bounded distance in $Y$, and in fact up to multiplying on the left by an element of $h_j$, we can assume that $\dx(1,g)$ is uniformly bounded. In particular, by Lemma \ref{lem:finitenessdoublecosets}, there are finitely many choices for $g$. Since $\rho_{gH_1}(H_j)$ is the coarse intersection of $gH_1$ and $H_j$, by Proposition \ref{prop:algint=geoint} we have that $\rho_{gH_1}(H_j)$ is within bounded Hausdorff distance of the intersection of $H_1^g$ and $H_j$, as required.
 \end{proof}

 Let $\{q_i\}$ be intersection-compatible homogeneous quasimorphisms. By the inductive hypothesis there exists a homogeneous quasimorphism $q'$ on $G$ that extends $q_2,\dots,q_l$.
 
 Set $q'_1=q_1-q'|_{H_1}$, so that $q'_1$ is a quasimorphism on $H_1$. In particular, by intersection-compatibility, $q'_1|_{K_j}$ vanishes for each $j$, and hence $q'_1|_{B_{C}(1)K_jB_C(1)}$ is bounded, where $B_C(1)$ is the ball of radius $C$ around $1$ in $H_1$. Hence, we can perturb $q'_1$ up to bounded error into an alternating quasimorphims $q''_1$ to ensure that $q''_1(h)=0$ whenever $h\in B_{C}(1)K_jB_C(1)$ for some $j$.
 
 Finally, we can set $q=\Psi(q''_1)+q'$. Notice that $q|_{H_1}$ is within bounded distance from $q_1$. Moreover, from Lemma \ref{lem:proj_is_peripheral} and the definition of $\Psi$ we see that $\Psi(q''_1)|_{H_j}$ is identically $0$ for all $j\geq 2$. Hence, again for $j\geq2$, $q|_{H_j}$ is within bounded distance of $q'$, whence of $q_j$, as required.
\end{proof}

\subsection{Adding a hyperbolically embedded subgroup}

By \cite[Theorem 6.14]{DGO}, any acylindrically hyperbolic group $G$ contains a unique maximal finite normal subgroup $E(G)$. Let $F_2$ denote a non-abelian free group of rank 2.

\begin{prop}\label{prop:add_subgroup}
 Let $G$ be a group and let $H_1,\dots,H_l$ be subgroups of infinite index with $H_i\h (G,X_i)$ with $\Gamma(G,X_i)$ hyperbolic. Then there exists a hyperbolically embedded subgroup $H$ so that $\{H,H_i\}\h (G,X_i)$ for every $i$ so that $H\cong F_2\times E(G)$, In particular $H\cap H_i^g$ is finite for every $g\in G$ and every $i$.
 %Moreover, if each $H_i$ is hyperbolically embedded in $(G,X)$ and $\Gamma(G,X)$ is hyperbolic, then we can choose $H$ so $\{H,H_i\}\h (G,X)$ for each $i$.
\end{prop}

\begin{proof}
  Notice that $G$ contains a finite-index subgroup $G'$ that acts trivially by conjugation on $E(G)$.

 By \cite[Theorem 5.4, Lemma 5.12]{Osin}, for every $i$ there exists $Y_i\supseteq X_i$ so that $H_i\h (G,Y_i)$ and $G$ acts on the hyperbolic space $\Gamma(G,Y_i\cup H_i)$ acylindrically and non-elementarily, and clearly $H_i$ has bounded orbits. By \cite{MaherSisto}, given independent random walks $(W_n),(Z_n)$ on $G'$, the probability that $\langle W_n,Z_n, E(G)\rangle$ is $F_2\times E(G)$ and hyperbolically embedded in $(G,Y_i\cup H_i)$ goes to $1$ as $n$ tends to infinity. Since  $X_i\subseteq Y_i\cup H_i$, the identity map from $(G,d_{X_i})$ to $(G,d_{Y_i\cup H_i})$ is Lipschitz. From this and Theorem \ref{thm:addQ}, it is easily seen that whenever $K$ is hyperbolically embedded in $(G,Y_i\cup H_i)$, we have that $\{K,H_i\}$ is hyperbolically embedded in $(G,X_i)$, and the proof is complete.
\end{proof}

 \begin{cor}\label{EHbnot_injective}
   Let $G$ be an acylindrically hyperbolic group and let $H_1,\dots,H_l\h (G,X)$ be subgroups of infinite index, where $\Gamma(G,X)$ is hyperbolic. Then the product of the restriction maps $res:EH^2_b(G,\mathbb{R})\to \Pi_{i=1}^l EH^2_b(H_i,\mathbb{R})$ is not injective.
 \end{cor}

 \begin{proof}
  Let $H$ be as in Proposition \ref{prop:add_subgroup}, and let $\phi$ be any non-trivial homogeneous quasimorphism on $H$. Then by Theorem \ref{compatible_extension} there exists a homogeneous quasimorphism $\overline \phi$ on $G$ that restricts to $0$ on each $H_i$ but restricts to $\phi$ on $H$. The class $\alpha\in EH^2_b(G,\mathbb{R})$ represented by $\overline\phi$ has $res(\alpha)=0$, but $\alpha$ is non-trivial.
 \end{proof}

 \section{Further Applications}
 A number of applications of finite width and bounded packing can be found in the literature. 
 In this Section, we mention a couple of these consequences for stable subgroups. Note that we have listed a number of examples of stable subgroups in the introductory section (Corollary \ref{cor:egs}).\\
 
 \noindent {\bf (1) Dimension of 
 cubulation} \\ In \cite{sageev1, sageev2} (see also \cite{hruska-wise}), Sageev proves the following:
 
 \begin{prop}
 	\label{fincub}
 	Suppose
 $	H$
 	is a finitely generated  codimension $1$
 	subgroup
 	of a finitely generated group
 	$G$ satisfying bounded packing.  Then the corresponding
 	CAT(0)
 	cube complex
 	is finite dimensional.
 \end{prop}
 
 Combining with Theorem \ref{widthpack} we have
 as an immediate Corollary:
 
 \begin{cor} \label{fincubcor}	Suppose
 	$	H$
 	is a finitely generated stable codimension $1$
 	subgroup
 	of a finitely generated group
 	$G$.  Then the corresponding
 	CAT(0)
 	cube complex
 	is finite dimensional.
 \end{cor}

 \medskip

 \noindent {\bf (2) Relative Rigidity} \\ In \cite{mahan-relrig}, the first author showed that a uniformly proper
 bijection between configurations of weak hulls of limit sets of quasiconvex subgroups of hyperbolic
 groups comes from a quasi-isometry of the groups, thus extending a result of Schwartz \cite{schwarz-inv}.
 The same proof works here once we prove the following strengthening of bounded packing, which provides a 
 "coarse Helly property" for stable subsets:
 
 \begin{proposition}[Coarse centers]\label{cbc}
 Let $\HH =\{H_1, \cdots, H_l\}$ be a finite collection of 
 stable subgroups of a finitely generated group $G$. For each $D \ge0$ there exists $R\ge0$ such that if $\mathcal{C}$ is a collection of pairwise $D$--close cosets of subgroups of $\HH$, then  there is an $x \in G$ such that,
 \[
 N_R(x) \cap gH_i \neq \emptyset
 \]
 for each $gH_i \in \mathcal{C}$.
 \end{proposition}
 
 \begin{proof}
Fix a Cayley graph $\Gamma$ for $G$ and for simplicity assume that $\Gamma$ contains Cayley graphs of the $H_i$ as subgraphs.
By Theorem \ref{widthpack}, $\#\mathcal{C} \le N$ for some constant $N\ge 0$ depending only on $D$. 

Consider the connected subgraph of $\Gamma$,
\[
Y = \bigcup N_D(gH_i),
\]
where the union is over all $gH_i \in \mathcal{C}$.
\begin{claim}
The subgraph $Y$ is $f'$--stable in $X$, where $f'$ depends only on $D$ and the stability constants of the subgroups in $\HH$.
\end{claim}

Let us see how the proposition follows from the claim. Since $Y$ is stable in $X$, $Y$ is $\delta$--hyperbolic, with $\delta$ depending only on $f'$ (\cite[Lemma 3.3]{DT15}) and the subgraphs $gH_i$ are quasiconvex. Hence, we may apply either \cite[Lemma 7]{niblo2003coxeter} or \cite[Lemma 3.3]{mahan-relrig} which give an $R\ge0$, depending on $N$ and $\delta$, and an $x\in Y$ such that the $R$--ball about $x$ in $Y$ meets each $gH_i \in \mathcal{C}$. This will complete the proof.

Hence, it remains to establish the claim. We do so by induction on $N = \#\mathcal{C}$. For $N = 1$, there is nothing to prove. In general, write $Y = Y_N \cup Z$ where $Z = N_D(gH_i)$ for some $gH_i \in \mathcal{C}$ and $Y_N$ is the union of $D$--neighborhoods of the cosets in $\mathcal{C} \setminus \{gH_i\}$. Note that $Y_N$ is $f'$--stable by the induction hypothesis and $Z$ is stable with stability function depending only on $D$ and that of $H_i$. Let $a,b \in Y$ and assume (as we may) that $a \in Y_N$ and $b\in Z$. Further, pick any $c \in Y_N \cap Z$ and choose geodesics $\alpha = [a,c]$, $\beta = [b,c]$, and $\gamma = [b,c]$. By stability of $Y_N$ and $Z$, $\alpha$ and $\beta$ are uniformly stable and we conclude from Lemma \ref{lem:cordes}, that $\gamma$ is contained in a uniformly bounded neighborhood of $Y$ and is uniformly stable. Hence, we conclude that $Y$ is uniformly stable in $\Gamma$. (Note that ``uniform'' here depends on $N$, which is bounded by Theorem \ref{widthpack}.) 
 \end{proof}
 
 Let $G_1, G_2$ be finitely generated groups
 with Cayley graphs $\Gamma_1, \Gamma_2$, word metrics $d_1, d_2$,
 and stable subgroups $H_1, H_2$. Let $\Lambda_1$, $\Lambda_2$ be the limit sets of $H_1, H_2$ in $\partial_M G_1, \partial_M G_2$ respectively. Let $\JJ_i$,
 $i=1,2$, be the collection  of translates of $H_{wi}$, the weak hulls
  of $\Lambda_i$ in $\Gamma_i$. Each $d_i$ induces a pseudo-metric  on the collection $\JJ_i$ for $i = 1, 2$
  by regarding $J, K \in \JJ_i$ as closed subsets of $\Gamma_i$. We continue calling this induced metric $d_i$.
  
  \begin{defn} A bijective set map $\phi$ from $\JJ_1 \rightarrow \JJ_2$  is said to be 
  	{\bf uniformly proper } if there exists a function $f: \natls \rightarrow \natls$ such that 
  	\begin{enumerate}
  		\item For all $J, K \in \JJ_1$,
  		$d_{1} (J,K)) \leq n \Rightarrow d_{2} (\phi(J) ,\phi(K) \leq f(n)$,
  		\item For all $J, K \in \JJ_2$,
  		$d_{2} (J,K)) \leq n \Rightarrow d_{1} (\phi^{-1}(J) ,\phi^{-1}(K) \leq f(n)$.
  	\end{enumerate}
  \end{defn}

\begin{defn}
	A map $f$ from $\Gamma_1$ to $\Gamma_2$ is said to pair the sets $\JJ_1$ and $\JJ_2$ as $\phi$ does if there exists a function $h: \natls \to \natls$ such that for all $p \in \Gamma$, $J \in \JJ_1$,
	$d_1 (p ,J) \leq n \Rightarrow d_2 (f(p) ,
	\phi (J)) \leq h(n)$.
\end{defn}

The axioms of \cite[Section 3.4]{mahan-relrig} are now verified exactly as in that paper. (The main point in \cite{mahan-relrig} is the existence of a coarse barycenter, which in our context is established by Proposition \ref{cbc}.)
 We thus have the following:
  
  \begin{prop}\label{relrigprop}
  	Let $G_1, G_2$, $\Gamma_1, \Gamma_2$,  $d_1, d_2$ and $\JJ_1, \JJ_2$ be as above. Let
  	$\phi: \JJ_1 \to \JJ_2$ be uniformly proper.
  	Then there exists a quasi-isometry $q$ from $\Gamma_1$ to $\Gamma_2$ which pairs the sets $\JJ_1$ and $\JJ_2$ as $\phi$ does (in particular, $\Gamma_1, \Gamma_2$ are quasi-isometric).
  \end{prop}

\bibliographystyle{alpha}
\bibliography{width4stable.bbl}

\end{document}